\documentclass[12pt]{amsart}
\usepackage{amsmath,amsthm,amsfonts,amscd,amssymb,eucal,latexsym,mathrsfs,appendix,yhmath,setspace}
\usepackage[numbers,sort&compress]{natbib}
\usepackage[all,cmtip]{xy}
\usepackage{enumerate}
\usepackage{color}
\newtheorem{theorem}{Theorem}[section]

\newtheorem{lemma}[theorem]{Lemma}
\newtheorem{proposition}[theorem]{Proposition}

\theoremstyle{definition}

\newcommand{\h}{{\text{\rm h}}}
\newcommand{\topol}{{\text{\rm top}}}

\newcommand{\mdim}{{\rm mdim}}
\newcommand{\ord}{{\rm ord}}

\newcommand{\cA}{{\mathcal A}}

\newcommand{\cE}{{\mathcal E}}
\newcommand{\cF}{{\mathcal F}}

\newcommand{\cM}{{\mathcal M}}

\newcommand{\cQ}{{\mathcal Q}}

\newcommand{\cU}{{\mathcal U}}

\newcommand{\cW}{{\mathcal W}}

\newcommand{\Nb}{{\mathbb N}}

\newcommand{\Rb}{{\mathbb R}}

\newcommand{\diam}{{\rm diam}}

\newcommand{\supp}{{\rm supp}}

\newcommand{\IE}{{\rm IE}}

\newcommand{\Widim}{{\rm Widim}}

\allowdisplaybreaks

\begin{document}

\title
[Relative topological entropy and relative mean dimension]
{Relative topological entropy and relative mean dimension of induced factors}

\author[Kairan Liu]{Kairan Liu}
\address{Kairan Liu: College of Mathematics and Statistics, Chongqing University, Chongqing 401331, China}
\email{lkr111@cqu.edu.cn}

\author[Yixiao Qiao]{Yixiao Qiao}
\address{Yixiao Qiao (Corresponding author): School of Mathematics and Statistics, HNP-LAMA, Central South University, Changsha, China}
\email{yxqiao@mail.ustc.edu.cn}


\begin{abstract}
We study the relation of relative topological entropy and relative mean dimension between a factor map and its induced factor map for amenable group actions. On the one hand, we prove that a factor map has zero relative topological entropy if and only if so does the induced factor map. On the other hand, we prove that a factor map has positive relative topological entropy if and only if the induced factor map has infinite relative mean dimension. 
\end{abstract}

\maketitle

\section{Introduction}
The study of various complexities and their relations is one of the most central and significant topics in the field of dynamical systems. Among all the topological conjugacy invariants, topological entropy \cite{AKM65} and mean dimension \cite{Gro99,LW00} attract remarkable attention. Intuitively speaking, topological entropy and mean dimension measure how many bits and parameters, respectively, per iterate we need to describe an orbit in a dynamical system. Although properties for topological entropy may be not preserved by mean dimension, it is worth mentioning that mean dimension can distinguish infinite topological entropy and infinite dimensional (i.e., with the state space being infinite dimensional) systems, and has been further applied to study some problems in dynamical systems which seem difficult to be touched within the framework of entropy theory, for example, dynamical embedding problems \cite{Gut15,GQT19,GT14,GT20,JQhilbert,Lin99,LT14,LW00}.

The induced system $(\cM(X),\widetilde{T})$ of a dynamical system $(X,T)$ was well investigated by Bauer and Sigmund \cite{BS75}, where $\cM(X)$ denotes the space of all Borel probability measures on the compact metrizable space $X$, being equipped with the weak$^\ast$-topology, and where $\widetilde{T}:\cM(X)\to\cM(X)$ denotes the push-forward map (which is again a homeomorphism) induced by the homeomorphism $T:X\to X$. Bauer and Sigmund \cite{BS75} (in 1975) initially showed that if the topological entropy of a system is positive, then the topological entropy of its induced system has to be infinite. Glasner and Weiss \cite{GW95} (in 1995) further developed their connection, and proved a remarkable result: if a system has zero topological entropy then so does its induced system (note that the converse of this statement is obviously true). Later, Kerr and Li \cite[Theorem 5.10]{KL05} (in 2005) showed that a system is null (i.e., with topological sequence entropy being zero along all the sequences) if and only if so is its induced system. These results were also revisited by Qiao and Zhou \cite{QZ17} (in 2017) where it was proven that for any fixed sequence, a system has zero topological sequence entropy if and only if so does its induced system, which thus strengthens and unifies the above-mentioned results due to Glasner--Weiss \cite{GW95} and Kerr--Li \cite{KL05}. Furthermore, as an application it was shown in Qiao--Zhou \cite{QZ17} that the entropy dimension (introduced originally and developed systematically by Dou, Huang, and Park \cite{DHP11} as a topological conjugacy invariant distinguishing zero entropy systems, for a deep exploration of the measure-theoretical version of entropy dimension see Dou--Huang--Park \cite{DHP19}) of a system is equal to that of its induced system. Recently, Burguet and Shi \cite{BS25} showed that if a system has positive topological entropy then its induced system has infinite mean dimension (note that the converse also holds by reason of \cite{GW95,LW00}, and the statement holds true for amenable group actions \cite{SZ25}).

Relative topological entropy, as a generalization of topological entropy, reflects the complexity of fibers along factor maps. In his famous work, Bowen \cite{Bow73} (in 1973) established an inequality for (relative) topological entropy, which connects a system to its factor with all the fibers along a factor map. More precisely, Bowen \cite{Bow73} proved that the topological entropy of a system is bounded from above by the sum of the topological entropy of its factor and the relative topological entropy along a factor map. Note that Bowen's inequality for topological entropy has the same spirit as what was discovered in (classical) dimension theory (known as the Hurewicz theorem \cite{HW41}). Similarly, one might consider and further expect a reasonable analogue for mean dimension. However, Tsukamoto \cite{Tsu25} proved amazing results for relative mean dimension, which deny such a parallel expectation in the context of mean dimension. Besides, notice that still other phenomena in relation to these invariants are different from each other. For instance, unlike topological entropy, the mean dimension of a dynamical system does not necessarily decrease when taking factors, and moreover, we do not have a general product formula for mean dimension (but with the exception of self-products, for details we refer to \cite{LW00,Tsu19,JQ}).

All those aspects stated above naturally motivate us to study the complexity of factor maps from a mixed perspective. In the present paper, relative topological entropy, relative mean dimension, induced systems and induced factor maps come together, being investigated for all the actions of amenable groups. Our main aim is to prove the following relation among them:

\begin{theorem}\label{thm:main}
Let $G$ be a countably inﬁnite amenable group, $\pi:(X,G)\to(Y,G)$ a factor map between two $G$-systems, and $\widetilde{\pi}:(\cM(X),G)\to(\cM(Y),G)$ the induced factor map. Then we have:
\begin{itemize}
\item[(i)]
$\pi$ has zero relative topological entropy if and only if so does $\widetilde{\pi}$.
\item[(ii)]
$\pi$ has positive relative topological entropy if and only if $\widetilde{\pi}$ has infinite relative mean dimension.
\end{itemize}
\end{theorem}

The key ingredient of this paper is the local entropy theory, typically with the tool of independence sets \cite{HY06,KL07} being generally employed in our main proof. But a novelty here is that we develop a new approach (specifically provided in Section 3) to the goal with the help of combinatorial independence, and in particular, our method strengthens the results of Glasner--Weiss \cite{GW95}, Burguet--Shi \cite{BS25} and Shi--Zhang \cite{SZ25}.

This paper is organized as follows. In Section 2, we recall necessary notions and notations, and prepare some elementary properties for relative mean dimension. In Section 3 and Section 4, we prove Theorem \ref{thm:main}.

\subsection*{Acknowledgements}
The first-named author was supported by NNSF of China (Grant No. 12301224) and China Postdoctoral Science Foundation (Grant No. 2022M710527). The authors were supported by NNSF of China (Grant No. 12371190).

\section{Preliminaries}
In this section, we introduce basic notions and prepare some results which will be used in the sequel. Denote by $\mathbb{N}$ and $\mathbb{R}$ the set of positive integers and the set of real numbers, respectively. For $n\in\mathbb{N}$ we write $[n]$ for the set $\{1,2,\dots,n\}$ and $|F|$ for the cardinality of a set $F$. For every real number $r$, we denote by $\lfloor r \rfloor$ the greatest integer not exceeding $r$.

\subsection{$G$-systems and amenable groups}
Let $G$ be a countable (discrete) group. By a \textbf{$G$-system} $(X,G)$ we mean a compact metrizable space $X$ with the group $G$ acting on $X$ by homeomorphisms, that is, there is a continuous map $\Gamma:G\times X\to X$ satisfying that $\Gamma(e_G,x)=x$ and $\Gamma(g,\Gamma(h,x))=\Gamma(gh,x)$ for all $g,h\in G$ and $x\in X$. For simplicity $\Gamma(g,x)$ is usually written as $gx$.

We say that a countable group $G$ is \textbf{amenable} if there exists a \textbf{F{\o}lner sequence} of $G$, namely, a sequence $\{G_n\}_{n=1}^\infty\subset\mathcal{F}(G)$ such that $\lim_{n\to\infty}\frac{|(G_n\setminus gG_n)\cup(gG_n\setminus G_n)|}{|G_n|}=0$ for any $g\in G$, where $\mathcal{F}(G)$ denotes the collection of nonempty finite subsets of $G$. Throughout this paper $G$ is always assumed to be a countably infinite amenable group.

\subsection{Induced systems and induced factors}
Let $X$ be a compact metrizable space, and $\cM(X)$ the space of Borel probability measures on $X$ endowed with the weak$^\ast$-topology. It was classically proven that $\cM(X)$ also comes to be a compact metrizable space \cite[Theorem 6.4]{Par67}. Let $d$ be a compatible metric on $X$. The $1$-Wasserstein distance on $\cM(X)$ induced by $d$, defined by
\begin{equation}\label{Wass-distance-1}
W_d(\mu,\nu)=\inf_{\gamma\in\Gamma(\mu,\nu)}\int d(x,y)d\gamma(x,y),\;\;\;(\mu,\nu\in\cM(X))
\end{equation}
is compatible with the weak$^\ast$-topology, where $\Gamma(\mu,\nu)$ denotes the collection of measures on $X\times X$ with marginals $\mu$ and $\nu$ on the first and second factors respectively. For convenience, when the choice of metric $d$ is clear, we suppress $d$ and write $W(\mu,\nu)$. The Kantorovich-Rubinstein dual representation of $W$ is given by
$$W(\mu,\nu)=\sup_{f}\left|\int fd\mu-\int fd\nu\right|,\;\;\;(\mu,\nu\in\cM(X)),$$
where $f:X\to\mathbb{R}$ runs over all $1$-Lipschitz functions. For every metric space $(Y,\rho)$ let $\diam(Y,\rho)=\sup_{x,y\in Y}\rho(x,y)$ be the diameter of $Y$ with respect to the metric $\rho$. By \eqref{Wass-distance-1} one has
\begin{equation}\label{diam-W}\diam(\cM(X),W)\le\diam(X,d).\end{equation}

A $G$-system $(X,G)$ naturally induces a $G$-system $(\cM(X),G)$, which is defined by $g\mu=\mu\circ g^{-1}$ for all $\mu\in\cM(X)$ and $g\in G$. We call $(\cM(X),G)$ the \textbf{induced system} of $(X,G)$. Let $(X,G)$ and $(Y,G)$ be two $G$-systems. We say that $\pi:(X,G)\to(Y,G)$ is a \textbf{factor map} if $\pi:X\to Y$ is continuous and surjective, and satisfies $\pi(gx)=g\pi(x)$ for all $x\in X$ and $g\in G$. A factor map $\pi:(X,G)\to(Y,G)$ naturally induces a factor map $\widetilde{\pi}:(\cM(X),G)\to(\cM(Y),G)$, called the \textbf{induced factor map}, given by $\widetilde{\pi}:\cM(X)\to\cM(Y)$, $\mu\mapsto\mu\circ\pi^{-1}$.

\subsection{Relative topological entropy and independence}
Let $(X,G)$ be a $G$-system and $d$ a compatible metric on $X$. For any $H\in\cF(G)$ we define a compatible metric $d_H$ on $X$ as follows: $d_H(x,x^\prime)=\max_{s\in H}d(sx,sx^\prime)$, where $x,x^\prime\in X$.

For a closed subset $A$ of $X$, $\varepsilon>0$, and $H\in\cF(G)$, a subset $S\subset A$ is called \textbf{$(A,\varepsilon,d_H)$-separated} if for any distinct $x,y\in S$ one has $d_H(x,y)>\varepsilon$. Denote by $\sharp(A,\varepsilon,d_H)$ the largest cardinality among all the $(A,\varepsilon,d_H)$-separated subsets of $A$. Note that it has to be a finite number by the compactness of $A$.

Let $\pi:(X,G)\to(Y,G)$ be a factor map between two $G$-systems and $\{G_n\}_{n=1}^\infty$ a F{\o}lner sequence of $G$. The \textbf{relative topological entropy} of $\pi$ is defined by $$\h_\topol(\pi,G)=\lim_{\varepsilon\to0}\limsup_{n\to\infty}\frac{\log(\sup_{y\in Y}\sharp(\pi^{-1}(y),\varepsilon,d_{G_n}))}{\vert G_n\vert}.$$ It is independent of the choice of F{\o}lner sequences of $G$ and the compatible metric $d$ on $X$.

For two subsets $U_1$ and $U_2$ of $X$ we say that $I\subset G$ is an \textbf{independence set for $(U_1,U_2)$ with respect to $\pi$} if for every nonempty finite subset $J\subset I$ there exists $y\in Y$ such that for every $\sigma\in[2]^J$ one has $$\pi^{-1}(y)\cap(\cap_{h\in J}h^{-1}U_{\sigma(h)})\ne\emptyset.$$ We denote by $\mathcal{P}^\pi(U_1,U_2)$ the collection of independence sets for $(U_1,U_2)$ with respect to $\pi$. We say that $(U_1,U_2)$ has \textbf{positive independence density with respect to $\pi$} if there is some $c>0$ such that for every $H\in\mathcal{F}(G)$ there exists a subset $I$ of $H$ such that $I\in\mathcal{P}^\pi(U_1,U_2)$ and $\vert I\vert>c\vert H\vert$.

We call $(x_1,x_2)\in X\times X$ an \textbf{$\IE$-pair with respect to $\pi$} if for every neighbourhood $U_1\times U_2$ of $(x_1,x_2)$, $(U_1,U_2)$ has positive independence density with respect to $\pi$. Denote by $\IE(\pi,G)$ the set of all $\IE$-pairs with respect to $\pi$. The following result is a parallel generalization of the classical version. The proof can be found in \cite[Theorem 12.19]{KL16}), it is omitted here.  
\begin{theorem}\label{thm:positiveentropy}
Let $\pi:(X,G)\to(Y,G)$ be a factor map. Then the following statements are equivalent:
\begin{enumerate}
\item [(1)] $\h_\topol(\pi,G)>0$;
\item [(2)] $\IE(\pi,G)\setminus\Delta_2(X)\neq\emptyset$, where $\Delta_2(X)=\{(x,x):x\in X\}$;
\end{enumerate}
\end{theorem}

\subsection{Relative mean dimension}
Let $X$ and $P$ be compact metrizable spaces. Take a metric $d$ on $X$ compatible with its topology. For any $\varepsilon>0$, a continuous map $f:X\to P$ is called an \textbf{$\varepsilon$-embedding} with respect to $d$ if $f(x_1)=f(x_2)$ implies $d(x_1,x_2)<\varepsilon$, for all $x_1,x_2\in X$. Define $\Widim_\varepsilon(X,d)$ to be the smallest integer $n$ such that there is an $n$-dimensional simplicial complex $P$ admitting an $\varepsilon$-embedding $f:X\to P$ with respect to $d$.

Let $\pi:(X,G)\to(Y,G)$ be a factor map between two $G$-systems, $d$ a compatible metric on $X$, and $\{G_n\}_{n=1}^\infty$ a F{\o}lner sequence of $G$. The \textbf{relative mean dimension of $\pi$} is defined by $$\mdim(\pi,G)=\lim_{\varepsilon\to0}\lim_{N\to\infty}\frac{\sup_{y\in Y}\Widim_\varepsilon(\pi^{-1}(y),d_{G_N})}{\vert G_N\vert}.$$ It is easy to see that for every $\varepsilon>0$ and $E,H\in\cF(G)$$$\sup_{y\in Y}\Widim_\varepsilon(\pi^{-1}(y),d_{E\cup H})\le\sup_{y\in Y}\Widim_\varepsilon(\pi^{-1}(y),d_E)+\sup_{y\in Y}\Widim_\varepsilon(\pi^{-1}(y),d_H).$$ So the above limits exist, and the defined value does not depend upon the choice of F{\o}lner sequences of $G$ \cite[Theorem 4.38]{KL16}. It is also independent of the compatible metric $d$ on $X$. When the factor map $\pi:(X,G)\to(Y,G)$ is trivial (i.e., $Y$ is a singleton) this defines the mean dimension of $(X,G)$ (denoted by $\mdim(X,G)$). Clearly, any finite-to-one factor map has zero relative mean dimension.

As follows is an alternative approach to defining relative mean dimension. For any finite open cover $\alpha$ of $X$ define $\ord(\alpha)=\max_{x\in X}\sum_{U\in\alpha}1_U(x)-1$ and $\mathcal{D}(\alpha)=\min_{\beta\succ\alpha}\ord(\beta)$. Here the notation $\beta\succ\alpha$ means that $\beta$ refines $\alpha$. Recall that the topological dimension of $X$ is defined by $\dim(X)=\sup_\alpha\mathcal{D}(\alpha)$, where $\alpha$ runs over all finite open covers of $X$.

For finite open covers $\alpha$ and $\beta$ of $X$, $g\in G$, $H\in\cF(G)$, and a closed subset $A$ of $X$, set $g^{-1}\alpha=\{g^{-1}U:U\in\alpha\}$, $\alpha\vee\beta=\{U\cap V:U\in\alpha,V\in\beta\}$, $\alpha_H=\vee_{g\in H}g^{-1}\alpha$, $\alpha|_A=\{U\cap A:U\in\alpha\}$. We say that a continuous map $f:X\to Y$ is \textbf{$\alpha$-compatible} if there is a finite open cover $\beta$ of $f(X)$ such that $f^{-1}(\beta)\succ\alpha$.
\begin{lemma}[{\cite[Proposition 2.4]{LW00}}]\label{lem:widim-ord}
For a finite open cover $\alpha$ of $X$, $\mathcal{D}(\alpha)\le k$ if and only if there is a polyhedron $P$ with $\dim(P)=k$, which admits an $\alpha$-compatible continuous map $f:X\to P$.
\end{lemma}
\begin{proposition}\label{prop:mdim-widim-ord}
Let $\pi:(X,G)\to(Y,G)$ be a factor map and $\{G_n\}_{n=1}^\infty$ a F{\o}lner sequence of $G$. Then $$\mdim(\pi,G)=\sup_\alpha\lim_{n\to\infty}\frac{\sup_{y\in Y}\mathcal{D}(\alpha_{G_n}|_{\pi^{-1}(y)})}{\vert G_n\vert},$$ where $\alpha$ runs over all finite open covers of $X$.
\end{proposition}
\begin{proof}
First note that by Lemma \ref{lem:widim-ord} $$\sup_{y\in Y}\mathcal{D}(\alpha_{H\cup E}|_{\pi^{-1}(y)})\le\sup_{y\in Y}\mathcal{D}(\alpha_H|_{\pi^{-1}(y)})+\sup_{y\in Y}\mathcal{D}(\alpha_E|_{\pi^{-1}(y)})$$ for every finite open cover $\alpha$ of $X$ and $E,H\in\cF(G)$. So the above limit exists, and is independent of the choice of F{\o}lner sequences of $G$ \cite[Theorem 4.38]{KL16}.

To show $$\sup_\alpha\lim_{n\to\infty}\frac{\sup_{y\in Y}\mathcal{D}(\alpha_{G_n}|_{\pi^{-1}(y)})}{\vert G_n\vert}\le\mdim(\pi,G)$$ we take a finite open cover $\alpha$ of $X$ and a Lebesgue number $\lambda>0$ of $\alpha$ with respect to a compatible metric $d$ on $X$. It suffices to prove that for every $H\in\cF(G)$ and $y\in Y$,
\begin{equation}\label{eq:ordleqwidim}
\mathcal{D}(\alpha_H|_{\pi^{-1}(y)})\le\Widim_\lambda(\pi^{-1}(y),d_H).
\end{equation}
In fact, take a polyhedron $P$ with $\dim(P)=\Widim_\lambda(\pi^{-1}(y),d_H)$ and a $\lambda$-embedding $f:\pi^{-1}(y)\to P$ with respect to $d_H$. Then for every $g\in H$ and $p\in f(X)$, $\diam(g(f^{-1}(p)),d)<\lambda$ and hence $g(f^{-1}(p))$ is contained in some element $U\in\alpha$. Thus for every $p\in f(X)$, $f^{-1}(p)$ is contained in some element in $\alpha_H|_{\pi^{-1}(y)}$. For each $U\in\alpha_H|_{\pi^{-1}(y)}$ put $\widetilde{U}=\{p\in P:f^{-1}(p)\subset U\}$. Clearly, each $\widetilde{U}$ is open in $f(X)$, $f^{-1}(\widetilde{U})\subset U$, and all the sets $\widetilde{U}$ cover $f(X)$. Therefore $f$ is $\alpha_H|_{\pi^{-1}(y)}$-compatible, and by Lemma \ref{lem:widim-ord} we obtain \eqref{eq:ordleqwidim}.

To show $$\mdim(\pi,G)\le\sup_\alpha\lim_{n\to\infty}\frac{\sup_{y\in Y}\mathcal{D}(\alpha_{G_n}|_{\pi^{-1}(y)})}{\vert G_n\vert},$$ we take $\varepsilon>0$, $n\in\mathbb{N}$, $y\in Y$, and a finite open cover $\alpha$ of $X$ satisfying $\diam(U,d)<\varepsilon$ for all $U\in\alpha$. By Lemma \ref{lem:widim-ord} we can find a polyhedron $P$ with $\dim(P)=\mathcal{D}(\alpha_{G_n}|_{\pi^{-1}(y)})$, and a continuous map $f:\pi^{-1}(y)\to P$ which is $\alpha_{G_n}|_{\pi^{-1}(y)}$-compatible. So $\Widim_\varepsilon(\pi^{-1}(y),d_{G_n})\le\dim(P)$, and thus we conclude.
\end{proof}
A relative version of Lindenstrauss--Weiss \cite[Theorem 4.2]{LW00} still holds true, and the proof is omitted as is (almost) the same.
\begin{theorem}[{cf. \cite[Theorem 4.2]{LW00}}]\label{thm:finiteentropy}
Let $\pi:(X,G)\to(Y,G)$ be a factor map between two $G$-systems. If $\pi$ has finite relative topological entropy, then $\pi$ has zero relative mean dimension.
\end{theorem}

Notice that since $X$ and $Y$ can be viewed as subsystems of $\cM(X)$ and $\cM(Y)$, respectively, $\pi$ is the corresponding restriction of $\widetilde{\pi}$, the statements of Theorem \ref{thm:main} can be reduced to $h_{\topol}(\pi,G)=0\implies h_{\topol}(\widetilde{\pi},G)=0$ and $h_{\topol}(X,G)>0\iff\mdim(\widetilde{\pi},G)=\infty$. Moreover, by Theorem \ref{thm:finiteentropy} above, Theorem \ref{thm:main} can be further reduced to the derivations $\h_\topol(\pi,G)=0\implies\h_\topol(\widetilde{\pi},G)=0$ and $\h_\topol(\pi,G)>0\implies\mdim(\widetilde{\pi},G)=+\infty$.

\subsection{Simplices and the Lebesgue lemma}
Let $n\in\mathbb{N}$. An $(n-1)$-dimensional simplex is represented as $\Delta_n=\{(x_i)_{i\in[n]}\in[0,1]^n:\sum_{i\in[n]}x_i=1\}$. The topological dimension of $\Delta_n$ is still equal to $(n-1)$. An $l$-face of $\Delta_n$ is an $(l-1)$-dimensional simplex $\{(x_i)_{i\in I}\in[0,1]^I:\sum_{i\in I}x_i=1\}$ for some $I\subset[n]$ with $|I|=l$. For a face $F=\{(x_i)_{i\in I}\in[0,1]^I:\sum_{i\in I}x_i=1\}$ of $\Delta_n$ define its opposite face as $$\overline{F}=\{(x_i)_{i\in[n]\setminus I}\in[0,1]^{[n]\setminus I}:\sum_{i\in[n]\setminus I}x_i=1\}.$$ For every $k\in\mathbb{N}$ denote $\Delta_n^k=\Delta_n\times\cdots\times\Delta_n$ ($k$-times). For each face $F$ of $\Delta_n$ and $i\in[k]$ set $$F_i=\underbrace{\Delta_n\times\cdots\times\Delta_n}_{i-1}\times F\times\underbrace{\Delta_n\times\cdots\times\Delta_n}_{k-i}\subset\Delta_n^k.$$

The celebrated Lebesgue lemma is a key tool when establishing the theory of topological dimension of Euclidean spaces, which states that any finite open cover $\alpha$ of $[0,1]^n$, in which no elements meet two opposite faces of $[0,1]^n$, has $\ord(\alpha)\ge n$. For our purpose we borrow a tailored Lebesgue lemma verified by Burguet--Shi \cite{BS25}.
\begin{lemma}[{\cite[Lemma 9]{BS25}}]\label{lem:lebesgue}
Let $k,n\in\mathbb{N}$. If a finite open cover $\alpha$ of $\Delta_n^k$ satisfies the following condition:
\begin{itemize}\item
for any $i\in[k]$, any $m\in\mathbb{N}$, any $U_1,U_2,\dots,U_m\in\alpha$, and $(n-1)$-faces $F^1,F^2,\dots,F^m$ of $\Delta_n$ with $U_j\cap(F^j)_i\ne\emptyset$ for all $j\in[m]$, it holds that $(\cap_{j\in[m]}U_j)\cap(\overline{\cap_{j\in[m]}F^j})_i=\emptyset$,
\end{itemize}
then $\ord(\alpha)\ge nk$.
\end{lemma}

\section{Proof of $\h_\topol(\pi,G)=0\implies\h_\topol(\widetilde{\pi},G)=0$}
Let $\pi:X\to Y$ be a continuous surjective map between two compact metrizable spaces, and $\widetilde{\pi}:\cM(X)\to\cM(Y)$ the map induced by $\pi$ (being defined the same as the induced factor map, but note that we do not have any group actions for the moment). Let $H$ be a finite index set, and we set
$$R^H_{\widetilde{\pi}}=\big\{(\mu_\sigma)_{\sigma\in H}\in\prod_{\sigma\in H}\cM(X):\widetilde{\pi}(\mu_s)=\widetilde{\pi}(\mu_\omega)\ \text{for every}\ s,\omega\in H\big\}.$$ For every $n\in\mathbb{N}$ set $\cM_n(X)=\{\frac1n\sum_{i=1}^n\delta_{x_i}:x_i\in X\ \text{for all}\ i\in[n]\}$ and $$R^H_{\widetilde{\pi}_n}=\big\{(\mu_\sigma)_{\sigma\in H}\in\prod_{\sigma\in H}\cM_n(X):\widetilde{\pi}_n(\mu_s)=\widetilde{\pi}_n(\mu_\omega)\ \text{for every}\ s,\omega\in H\big\}$$ where $\widetilde{\pi}_n:\cM_n(X)\to\cM_n(Y)$ denotes the restriction of $\widetilde{\pi}$ to $\cM_n(X)$.
We shall borrow the following result. The case $\vert H\vert=2$ was proved by Liu--Wei (\cite[Lemma 4.1]{LW24}). For general $H$, the same argument applies, so the proof is omitted. 
\begin{lemma}\label{R_pi-dense} For every finite index set $H$,
$R_{\widetilde{\pi}}^H=\overline{\cup_{n\in\mathbb{N}}R^H_{\widetilde{\pi}_n}}$.
\end{lemma}

Let $E$ be a finite set and $F\subset E$. For every $\sigma\in [2]^E$, let
$\sigma\vert_F$ denote the restriction of $\sigma$ on $F$, that is, $\sigma\vert_F\in[2]^F$ and $\sigma\vert_F(s)=\sigma(s)$ for every $s\in F$. For every $S\subset [2]^E$ we denote by $S\vert_F$ the set $\{\sigma\vert_I\colon \sigma\in S\}$. We will need the following Sauer--Perles--Shelah lemma \cite{Sau72,She72}, the reader may refer to \cite[Lemma 12.14]{KL16} for the proof.
\begin{lemma}\label{KLind}
	Let $\lambda>1$, there exists a constant $c>0$ such that, for all $n\in\Nb$, if $S\subset[2]^{[n]}$ satisfies $\vert S\vert\geq \lambda^n$ then there is an $I\subset [n]$ with $\vert I\vert\geq cn$ and $S\vert_I=[k]^I$.
\end{lemma}

Let $E$ be a finite set such that $\vert E\vert$ is even. For every $\sigma\in[2]^E$ and $i\in[2]$ let $$E_i^\sigma=\{j\in E:\sigma(j)=i\},\;\;\;S_E=\{\sigma\in[2]^E:\vert E_1^\sigma\vert=\vert E_2^\sigma\vert\}.$$ We need to employ the following lemma (see \cite[Lemma 3.14]{LL25}). For completeness we include a proof.
\begin{lemma}\label{ind}
Let $d>0$ and $\tau>1/2$. Then there exists $c>0$, depending on $d$ and $\tau$, such that the following holds: for any finite set $E$ with $\vert E\vert$ even and large enough, and any $S_E^\prime\subset S_E$ with $\vert S_E^\prime\vert>d\cdot\vert S_E\vert$, if for each $\sigma\in S_E^\prime$ we take a subset $E_\sigma\subset E$ with $\vert E_\sigma\vert>\tau\cdot\vert E\vert$, then there exists $I\subset E$ with $\vert I\vert>c\cdot\vert E\vert$ such that for every $\omega\in[2]^I$ there is $\sigma\in S_E^\prime$ satisfying $\sigma\vert_I=\omega$ and $I\subset E_\sigma$.
\end{lemma}
\begin{proof}
Take $0<\delta<1/4$. Put $N=\vert E\vert$, $\cW=\{W\subset E:\vert W\vert=\lfloor\delta N\rfloor\}$, and $\cA=\{(\sigma,W)\in S_E^\prime\times\cW:W\subset E_\sigma\}$. For each $\sigma^\prime\in S_E^\prime$ one has $\vert\cA_{\sigma^\prime}\vert\ge{\tau N\choose\lfloor\delta N\rfloor}$, where $\cA_{\sigma^\prime}=\{(\sigma^\prime,W):(\sigma^\prime,W)\in\cA\}$. For every $W\in\cW$, we put $\cA_W=\{(\sigma,W):(\sigma,W)\in\cA\}$. Then $$\vert\cA\vert=\sum_{\sigma\in S_E^\prime}\vert\cA_\sigma\vert=\sum_{W\in\cW}\vert\cA_W\vert.$$ Hence there is some $W_0\in\cW$ such that
\begin{equation}\label{chp3-e-1}
\vert\cA_{W_0}\vert\ge\frac{\vert S_E^\prime\vert\cdot{\tau N\choose\lfloor\delta N\rfloor}}{\vert\cW\vert}=\frac{\vert S_E^\prime\vert\cdot{\tau N\choose\lfloor\delta N\rfloor}}{{N\choose\lfloor\delta N\rfloor}}.
\end{equation}
Thus, there exists $\sigma_0\in[2]^{E\setminus W_0}$ such that
\begin{equation}\label{chp3-e-2}
\vert\{(\sigma,W_0)\in\cA_{W_0}:\sigma\vert_{E\setminus W_0}=\sigma_0\}\vert\ge\frac{\vert\cA_{W_0}\vert}{2^{\vert E\setminus W_0\vert}}.
\end{equation}
Let $S_0=\{\sigma\vert_{W_0}:(\sigma,W_0)\in\cA_{W_0},\sigma\vert_{E\setminus W_0}=\sigma_0\}$. When $N\in\mathbb{N}$ is even and large enough, by \eqref{chp3-e-1}, \eqref{chp3-e-2}, and the Stirling approximation formula one has 
\begin{align*}
	\vert S_0\vert&\ge\frac{\vert\cA_{W_0}\vert}{2^{\vert E\setminus W_0\vert}}\ge\frac{d\cdot{N\choose N/2}\cdot{\tau N\choose\lfloor\delta N\rfloor}}{{N\choose\lfloor\delta N\rfloor}\cdot2^{\vert E\setminus W_0\vert}}\ge\frac{d\cdot (\tau N)!(N-\delta N)!}{(N/2)!(N/2)!(\tau N-\delta N +1)!\cdot  2^{\vert E\setminus W_0\vert}}\\
	&\ge\frac{d\cdot (\tau N/e)^{\tau N}\cdot \big((N-\delta N)/e\big)^{N-\delta N}}{N\cdot (N/2e)^N\cdot \big((\tau N-\delta N+1)/e\big)^{\tau N-\delta N+1}\cdot 2^{N-\delta N+1}}\\
\end{align*}
Since $(M+1)^{M+1}\le 4\cdot M^{M+1}$ for all $M>0$ sufficiently large, $\vert S_0\vert$ is bounded from below by
\begin{align*}
	 &\frac{d\cdot e\cdot 2^N\cdot (\tau N)^{\tau N}(N-\delta N)^{N-\delta N}}{N\cdot N^N \cdot 4(\tau N-\delta N)^{(\tau N-\delta N+1)}\cdot 2^{N-\delta N+1}}\\
	&=\frac{d\cdot e\cdot \tau^{\tau N}(1-\delta)^{(1-\delta)N}}{8 N^2(\tau-\delta)\cdot(\tau-\delta)^{(\tau-\delta)N}\cdot 2^{-\delta N}}\\
	&=\frac{d\cdot e}{8N^2(\tau-\delta)}\exp\big[N\big(-f(\tau)-f(1-\delta)+f(\tau-\delta)+\delta\log2\big)\big],
\end{align*}
 where $f(x)=-x\log x$. By noting that $$\lim_{\delta\to0^+}\frac{-f(\tau)-f(1-\delta)+f(\tau-\delta)+\delta\log2}{\delta}=\log(2\tau)>0$$ we can take $0<\delta<1/4$ sufficiently small such that $\frac{-f(\tau)-f(1-\delta)+f(\tau-\delta)+\delta\log2}{\delta}>\theta_1$ for some $0<\theta_1<\log(2\tau)$, and $0<\theta_2<\theta_1$ with $$\vert S_0\vert\ge\frac{d e}{8(\tau-\delta)N^2}\cdot e^{\theta_1\delta N}\ge2^{\theta_2\cdot\vert W_0\vert}$$ for every $N$ large enough. Then by Lemma \ref{KLind}, we can find some $c>0$ depending on $\theta_2$ only, such that when $N$ is large enough there exists $I\subset W_0$ satisfying $\vert I\vert>c\cdot\vert W_0\vert$ and $S_0\vert _I=[2]^I$. Hence we conclude.
\end{proof}
We also need an elementary lemma.
\begin{lemma}\label{diff}
For any distinct $\mu,\nu\in\cM(X)$ there exist open subsets $A,B$ of $X$ such that $\overline{A}\cap\overline{B}=\emptyset$ and $\mu(A)+\nu(B)>1$.
\end{lemma}
\begin{proof}
Since $\nu\neq\mu$, there exists a $f\in C(X)$ such that $\mu(f)\neq\nu(f)$. Note that $\nu,\mu$ are probability measures and $f^{-1}(r)\cap f^{-1}(s)=\emptyset$ for every distinct $r,s\in\Rb$. Then there exists $r\in\Rb$ such that $\mu(f^{-1}(r))=\nu(f^{-1}(r))=0$ and $\mu(f^{-1}(a,\infty))\neq \nu(f^{-1}(a,\infty))$. Without loss of generality, we can assume that $\mu(f^{-1}(a,\infty))> \nu(f^{-1}(a,\infty))$. It follows that
$$\mu(f^{-1}((r,+\infty)))+\nu(f^{-1}((-\infty,r)))>\nu(f^{-1}(\mathbb{R}))=1.$$
Since $\mu(f^{-1}r)=\nu(f^{-1}(r))=0$ and $f\in C(X)$, there exists $\delta>0$ small enough such that $A=f^{-1}((r+\delta,+\infty))$ and $B=f^{-1}((-\infty,r-\delta))$ are as required.
\end{proof}
We are now ready to prove:
\begin{proposition}
Let $\pi:(X,G)\to(Y,G)$ be a factor map between two $G$-systems and $\widetilde{\pi}:(\cM(X),G)\to(\cM(Y),G)$ the induced factor map. If $\h_\topol(\pi,G)=0$, then $\h_\topol(\widetilde{\pi},G)=0$.
\end{proposition}
\begin{proof}
Assume that $\h_\topol(\widetilde{\pi},G)>0$. We shall show $\h_\topol(\pi,G)>0$. By Theorem \ref{thm:positiveentropy}, $\IE(\widetilde{\pi},G)\setminus\Delta_2(\cM(X))\ne\emptyset$. Take $(\mu_1,\mu_2)\in\IE(\widetilde{\pi},G)\setminus\Delta_2(\cM(X))$. By Lemma \ref{diff}, there exist open subsets $A_1,A_2$ of $X$ such that $\overline{A_1}\cap\overline{A_2}=\emptyset$ and $\mu_1(A_1)+\mu_2(A_2)>1$. Again by Theorem \ref{thm:positiveentropy}, in order to show $\h_\topol(\pi,G)>0$ it suffices to show that $(A_1,A_2)$ has an independence set of positive density with respect to $\pi$.

Take $\delta>0$ sufficiently small such that $a_1+a_2>1$, where $a_1=\mu_1(A_1)-\delta>0$, $a_2=\mu_2(A_2)-\delta>0$. For $i\in[2]$ put $\cU_i=\{\theta\in\cM(X):\theta(A_i)>a_i\}$. Clearly, $\cU_1$ and $\cU_2$ are disjoint open subsets of $\cM(X)$, which satisfy $\mu_1\in\cU_1$ and $\mu_2\in\cU_2$ \cite[Chapter II, Theorem 6.1]{Par67}. Note that $\theta\in\overline{\cU_i}$ implies $\theta(\overline{A_i})\ge a_i$ for $i\in[2]$. Hence, $\overline{\cU_1}\cap\overline{\cU_2}=\emptyset$. Since $(\mu_1,\mu_2)\in\IE(\widetilde{\pi},G)$, there is some $c_0>0$ satisfying that for every $H\in\mathcal{F}(G)$ there exists $E\subset H$ with $\vert E\vert>c_0\cdot\vert H\vert$ such that $E$ is an independence set for $(\cU_1,\cU_2)$ with respect to $\widetilde{\pi}$.

Fix temporarily an independence set $E$ of $(\cU_1,\cU_2)$ with respect to $\widetilde{\pi}$, as mentioned above. Then there exists $\nu\in\cM(Y)$ such that for each $\sigma\in[2]^E$ there is $\lambda_\sigma\in\cM(X)$ satisfying
\begin{equation}\label{entropy-e-0}
\lambda_\sigma\in\widetilde{\pi}^{-1}(\nu)\cap(\cap_{h\in E}h^{-1}\cU_{\sigma(h)}).
\end{equation}
This means $(\lambda_{\sigma})_{\sigma\in[2]^E}\in R^{2^E}_{\widetilde{\pi}}$ and $\prod_{\sigma\in[2]^E}(\cap_{h\in E}h^{-1}\cU_{\sigma(h)})$ is a neighborhood of it. Then
by Lemma \ref{R_pi-dense}, we can assume that $(\lambda_{\sigma})_{\sigma\in[2]^E}\in R^{2^E}_{\widetilde{\pi}_L}$ for some $L\in\Nb$. That is, we can find $y_i\in Y$ and $x_i^\sigma\in X$ with $\pi(x_i^\sigma)=y_i$ for all $i\in[L]$ and $\sigma\in[2]^E$, such that
\begin{equation}\label{entropy-e-1}
\nu=\frac1L\sum_{i\in[L]}\delta_{y_i},\quad\lambda_\sigma=\frac1L\sum_{i\in[L]}\delta_{x_i^\sigma}.
\end{equation}
Without loss of generality, we assume that $\vert E\vert$ is even. For $E$, $i\in[2]$, and $\sigma\in[2]^E$, recall the notations $E_i^\sigma$ and $S_E$ as defined previously.

Next define a family of maps $\{\Psi_\sigma:X\to[0,2]\}_{\sigma\in S_E}$ as follows: $$\Psi_\sigma(x)=\frac{1}{\vert E_1^\sigma\vert}\sum_{h\in E_1^\sigma}1_{A_1}(hx)+\frac{1}{\vert E_2^\sigma\vert}\sum_{h\in E_2^\sigma}1_{A_2}(hx),\quad(x\in X).$$ It follows from \eqref{entropy-e-0} that $$\int_X\Psi_\sigma d\lambda_\sigma>a_1+a_2,\quad\forall\sigma\in S_E.$$ Put $b=\frac{a_1+a_2+1}{2}>1$, $d=\frac{a_1+a_2-1}{3-a_1-a_2}>0$, and $H_\sigma=\{x\in X:\Psi_\sigma(x)>b\}$ for $\sigma\in S_E$. It follows that $$a_1+a_2<\int_X\Psi_\sigma d\lambda_\sigma\le2\lambda_\sigma(H_\sigma)+b(1-\lambda_\sigma(H_\sigma)).$$ Thus, $\lambda_\sigma(H_\sigma)>d$. By \eqref{entropy-e-1}, $\vert W_\sigma\vert>dL$, where $W_\sigma=\{i\in[L]:\Psi_\sigma(x_i^\sigma)>b\}$.

For each $k\in[L]$ let $S_k=\{\sigma\in S_E:k\in W_\sigma\}$. Then $\sum_{k\in[L]}\vert S_k\vert=\sum_{\sigma\in S_E}\vert W_\sigma\vert>dL\cdot\vert S_E\vert$. Hence there exists $i_E\in[L]$ with $\vert S_{i_E}\vert>d\cdot\vert S_E\vert$. Put $\tau=b/2>1/2$ and, for every $\sigma\in S_{i_E}$, $E_\sigma=\{h\in E_1^\sigma:hx_{i_E}^\sigma\in A_1\}\cup\{h\in E_2^\sigma:hx_{i_E}^\sigma\in A_2\}$. Since $\vert E_1^\sigma\vert=\vert E_2^\sigma\vert$, one has $\vert E_\sigma\vert>\tau\cdot\vert E\vert$ for all $\sigma\in S_{i_E}$.

By Lemma \ref{ind} there is $c>0$ such that for $E$ with $\vert E\vert$ being sufficiently large, there is $I\subset E$ with $\vert I\vert>c\cdot\vert E\vert$ such that for every $\omega\in[2]^I$ there is $\sigma\in S_{i_E}$ such that $\sigma\vert_I=\omega$ and $I\subset E_\sigma$. More precisely, if $h\in I$ and $j\in[2]$ satisfy $\omega(h)=j$, then $h\in E_j^\sigma\cap E_\sigma$ and therefore $hx_{i_E}^\sigma\in A_j$. From $\pi(x_{i_E}^\sigma)=y_{i_E}$ we deduce that $x_{i_E}^\sigma\in\pi^{-1}(y_{i_E})\cap\cap_{h\in I}h^{-1}A_{\omega(h)}$. This implies that $I$ is an independence set for $(A_1,A_2)$ with respect to $\pi$ and $\vert I\vert>cc_0\cdot\vert H\vert$. By Theorem \ref{thm:positiveentropy} we conclude.
\end{proof}

\section{Proof of $\h_\topol(\pi,G)>0\implies\mdim(\widetilde{\pi},G)=+\infty$}
Let $G$ be a countably infinite amenable group, $\pi:(X,G)\to(Y,G)$ a factor map between two $G$-systems, and $\widetilde{\pi}:(\cM(X),G)\to(\cM(Y),G)$ the induced factor map. Suppose that $\h_\topol(\pi,G)>0$. We will show $\mdim(\widetilde{\pi},G)=+\infty$.

In then following, for a compact metric space $(Z,\rho)$, point $x\in Z$ and subset $V\subset Z$, let $\rho(x,V)$ denote $\inf\{\rho(x,y)\colon y\in V\}$. For two subsets $V_1,V_2\subset Z$, let $\rho(V_1,V_2)$ denote $\inf\{\rho(x,y)\colon x\in V_1,y\in V_2\}$.

By Theorem \ref{thm:positiveentropy}, there exists $(x_1,x_2)\in \IE(\pi,G)\setminus\Delta_2(X)$. Then there exist closed neighborhoods $V_1$ of $x_1$ and $V_2$ of $x_2$, respectively, such $V_1\cap V_2=\emptyset$ and $(V_1,V_2)$ has positive independence density with respect to $\pi$, i.e., there is some $r>0$ such that for each $E\in\mathcal{F}(G)$ there exists $J\subset E$ satisfying that $\vert J\vert>r\cdot\vert E\vert$ and that $J$ is an independence set for $(V_1,V_2)$ with respect to $\pi$.

Take $0<\eta<1/8$. Fix $H\in\mathbb{N}$ temporarily. Take $M\in\mathbb{N}$ with $r/4<H/M<r/2$. Take pairwise distinct elements $h_1,\dots,h_M\in G$. Take $\delta>0$ such that for any $z,z^\prime\in X$ one has
\begin{equation}\label{e-distance1}
d(z,z^\prime)\le\delta\implies d(h_jz,h_jz^\prime)<d(V_1,V_2)/2,\quad\forall j\in[M].
\end{equation}
Take $0<\varepsilon<\frac{\delta^2}{2\cdot\diam(X)\cdot2^H}$ and a F{\o}lner sequence $\{G_n\}_{n\in\mathbb{N}}$ of $G$. Then for every $n\in\mathbb{N}$ sufficiently large, $\vert G_n\cap(\cap_{j\in[M]}h_j^{-1}G_n)\vert>(1-\eta)\cdot\vert G_n\vert$. Put $T_n=\lfloor\frac{(1-\eta)\cdot\vert G_n\vert}{M^2}\rfloor$.
\subsection*{Claim 0}
For every sufficiently large $n\in\mathbb{N}$, there exists $\{g_1,\dots,g_{T_n}\}\subset G_n$ such that $E^\prime_{n,1},\dots,E^\prime_{n,T_n}$ are pairwise disjoint subsets of $G_n$, where $E^\prime_{n,\ell}=\{h_1g_\ell,\dots,h_Mg_\ell\}$ for $\ell\in[T_n]$.
\begin{proof}
Fix $n\in\mathbb{N}$ large enough. Let $G_n^0=\{g\in G_n:h_jg\in G_n,\forall j\in[M]\}$. Then $\vert G_n^0\vert>(1-\eta)\cdot\vert G_n\vert$. Take $g_1\in G_n^0$. Put $E^\prime_{n,1}=\{h_1g_1,\dots,h_Mg_1\}\subset G_n$. For each $2\le\ell\le T_n$ set $\widetilde{E}_{n,\ell-1}=\{h_j^{-1}h_{j^\prime}g_{\ell-1}:j,j^\prime\in[M]\}$, and take $g_\ell\in G_n^0\setminus\cup_{p=1}^{\ell-1}\widetilde{E}_{n,p}$. Let $E^\prime_{n,\ell}=\{h_1g_\ell,\dots,h_Mg_\ell\}$. Then $g_\ell$ and $E^\prime_{n,\ell}$ for $\ell\in[T_n]$ are as desired.
\end{proof}
Now for every sufficiently large $n\in\mathbb{N}$ take $E_n\subset\cup_{\ell=1}^{T_n}E^\prime_{n,\ell}$ with $\vert E_n\vert>r\cdot\vert\cup_{\ell=1}^{T_n}E^\prime_{n,\ell}\vert$ such that $E_n$ is an independence set for $(V_1,V_2)$ with respect to $\pi$, let $\cQ_n=\{\ell\in[T_n]:\vert E_n\cap E^\prime_{n,\ell}\vert\ge H\}$ and put $m_n=\vert\cQ_n\vert$. Then $$rT_nM=r\cdot\vert\cup_{\ell=1}^{T_n}E^\prime_{n,\ell}\vert<\vert E_n\vert=\sum_{\ell\in[T_n]}\vert E_n\cap E^\prime_{n,\ell}\vert\le m_n\cdot M+(T_n-m_n)\cdot H$$ which implies that $$m_n\ge T_n\cdot(rM-H)/(M-H)\ge T_n\cdot(r-H/M)\ge rT_n/2.$$ Without loss of generality assume $\cQ_n=[m_n]$. Take $h_{i,j}\in\{h_\ell:\ell\in[M]\}$ for every $i\in[M_n]$ and $j\in[H]$, such that $h_{i,j}\ne h_{i,j^\prime}$ for any $i\in[m_n]$ and distinct $j,j^\prime\in[H]$, and such that $E_{n,i}\subset E^\prime_{n,i}\cap E_n$, where $E_{n,i}=\{h_{i,1}g_i,\dots,h_{i,H}g_i\}$. Then there exists $y_{H,n}\in Y$ such that for every $\cE=(\cE_1,\dots,\cE_{m_n})\in([2]^H)^{m_n}$ there is some $x_\cE\in\pi^{-1}(y_{H,n})\cap\cap_{i\in[m_n]}\cap_{j\in[H]}(h_{i,j}g_i)^{-1}V_{\cE_i(j)}$. Next define $\Psi:\Delta_{[2]^H}^{m_n}\to\cM(X)$ by $\vec{t}=(t_1,\dots,t_{m_n})\mapsto\sum_{\cE\in([2]^H)^{m_n}}(\prod_{i\in[m_n]}t_i(\cE_i))\delta_{x_\cE}$. Clearly, $\Psi$ is a well-defined continuous injective map. Put $L_n=\Psi(\Delta_{[2]^H}^{m_n})$. Then $L_n\subset\widetilde{\pi}^{-1}(\delta_{y_{H,n}})$.

In order to complete the proof, by Proposition \ref{prop:mdim-widim-ord} it suffices to prove that for any sufficiently large $H\in\mathbb{N}$ there is a finite open cover $\alpha$ of $\cM(X)$ such that for all $n\in\mathbb{N}$ large enough $$\mathcal{D}(\vee_{s\in G_n}s^{-1}\alpha\vert_{\widetilde{\pi}^{-1}(\delta_{y_{H,n}})})\ge\frac{r^3}{4^4H^2}\cdot2^H\cdot\vert G_n\vert.$$ We need show three claims below, for all sufficiently large $n\in\mathbb{N}$. For each face $F$ of $\Delta_{[2]^H}$ and each $i\in[m_n]$ set $S_{F_i}=\cup_{\nu\in\Psi(F_i)}\supp(\nu)$.
\subsection*{Claim 1}
Let $i\in[m_n]$ and $\mu\in L_n$. For any face $F$ of $\Delta_{[2]^H}$ there exist $\mu'\in\Psi(F_i)$ and $\mu''\in\Psi(\overline{F}_i)$ such that $\mu=\mu(S_{F_i})\mu'+\mu(S_{\overline{F}_i})\mu''$.
\begin{proof}
Take $I\subset[2]^H$ such that $F=\{a\in\Delta_{[2]^H}:\sum_{\omega\in I}a(\omega)=1\}$. Suppose $\mu=\Psi(\vec{t})$, $\vec{t}=(t_1,\dots,t_{m_n})\in\Delta_{[2]^H}^{m_n}$. If $\vec{t}\in F_i\cup\overline{F}_i$, then the statement is obvious. Now assume $\vec{t}\notin F_i\cup\overline{F}_i$. Let $b_1=\sum_{\omega\in I}t_i(\omega)\ne0$, $b_2=\sum_{\omega\notin I}t_i(\omega)\ne0$. Then $$\mu(S_{F_i})=\mu(\{x_\cE:\cE=(\cE_1,\dots,\cE_{[m_n]})\in([2]^H)^{m_n},\cE_i\in I\})$$$$=\sum_{\cE\in([2]^H)^{m_n},\cE_i\in I}\prod_{j\in[m_n]}t_j(\cE_j)=b_1,$$ and similarly, $\mu(S_{\overline{F}_i})=b_2$.

Let $\vec{t'}=(t_1,\dots,t_i',\dots,t_{m_n})\in F_i$, $\vec{t''}=(t_1,\dots,t_i'',\dots,t_{m_n})\in\overline{F}_i$, where $$t_i'(\omega)=\begin{cases}\frac{1}{b_1}\cdot t_i(\omega),&\omega\in I\\0,&\omega\notin I\end{cases},\quad\quad t_i''(\omega)=\begin{cases}\frac{1}{b_2}\cdot t_i(\omega),&\omega\notin I\\0,&\omega\in I\end{cases}.$$ Then $\vec{t}=b_1\vec{t'}+b_2\vec{t''}$. By the definition of $\Psi$, one has $\mu=\Psi(\vec{t})=\mu(S_{F_i})\Psi(\vec{t'})+\mu(S_{\overline{F}_i})\Psi(\vec{t''})$.
\end{proof}
\subsection*{Claim 2}
For any $i\in[m_n]$, any $\mu\in L_n$, and any face $F$ of $\Delta_{[2]^H}$$$\delta\cdot\mu(S_{\overline{F}_i})\le W_{G_n}(\mu,\Psi(F_i))\le\diam(X)\cdot\mu(S_{\overline{F}_i}).$$
\begin{proof}
Take $I\subset[2]^H$ such that $F=\{a\in\Delta_{[2]^H}:\sum_{\omega\in I}a(\omega)=1\}$. For any $x\in S_{F_i}$ and $x'\in S_{\overline{F}_i}$ there exist $\cE,\cE'\in([2]^H)^{m_n}$ such that $\cE_i\in I$, $\cE'_i\notin I$, $x=x_\cE$ and $x'=x_{\cE'}$. Thus, there is some $j\in[H]$ such that $\cE_i(j)\ne\cE'_i(j)$ which implies that $h_{i,j}g_ix_\cE\in V_p$ and $h_{i,j}g_ix_{\cE'}\in V_{p'}$ for $p\ne p'\in[2]$. Hence $d(h_{i,j}g_ix,h_{i,j}g_ix')\ge d(V_1,V_2)$. By \eqref{e-distance1}, $d(g_ix_\cE,g_ix_{\cE'})>\delta$. It follows that
\begin{equation}\label{WMn1}
d(g_iS_{F_i},g_iS_{\overline{F}_i})\ge\delta.
\end{equation}
Thus, for all $\nu\in\Psi(F_i)$$$W_{G_n}(\mu,\nu)=\max_{s\in G_n}W(s\mu,s\nu)\ge W(g_i\mu,g_i\nu)$$$$\ge\left\vert\int_Xd(x,g_iS_{F_i})d\mu(g_i^{-1}x)-\int_Xd(x,g_iS_{F_i})d\nu(g_i^{-1}x)\right\vert$$$$=\left\vert\int_Xd(g_ix,g_iS_{F_i})d\mu(x)-\int_Xd(g_ix,g_iS_{F_i})d\nu(x)\right\vert$$$$\ge\mu(S_{\overline{F}_i})\cdot d(g_iS_{F_i},g_iS_{\overline{F}_i})\ge\delta\cdot\mu(S_{\overline{F}_i}),$$ and thus, $W_{G_n}(\mu,\Psi(F_i))\ge\delta\cdot\mu(S_{\overline{F}_i})$. This proves the former inequality. To see the latter, applying Claim 1, $\mu=\mu(S_{F_i})\mu'+\mu(S_{\overline{F}_i})\mu''$ for $\mu'\in\Psi(F_i)$ and $\mu''\in\Psi(\overline{F}_i)$. By \eqref{diam-W} $$W_{G_n}(\mu,\Psi(F_i))\le W_{G_n}(\mu,\mu')=\mu(S_{\overline{F}_i})\cdot W_{G_n}(\mu',\mu'')\le\mu(S_{\overline{F}_i})\cdot\diam(X).$$ This shows the latter inequality.
\end{proof}
\subsection*{Claim 3}
If $\beta$ is a finite open cover of $L_n$ such that $\sup_{B\in\beta}\diam(B,W_{G_n})<\varepsilon$, then $\ord(\beta)\ge\frac{r^3}{4^4H^2}\cdot2^H\cdot\vert G_n\vert$.
\begin{proof}
Since $\Psi:\Delta_{[2]^H}^{m_n}\to L_n$ is a homeomorphism, it is equivalent to showing that $\ord(\Psi^{-1}(\beta))\ge\frac{r^3}{4^4H^2}\cdot2^H\cdot\vert G_n\vert$.

First we shall prove that $\Psi^{-1}(\beta)$ satisfies the condition as stated in Lemma \ref{lem:lebesgue}, i.e., for $i\in[m_n]$, $k\in\mathbb{N}$, $U_1,\dots,U_k\in\beta$ with $\cap_{j\in[k]}U_j\ne\emptyset$, and $(2^H-1)$-faces $F^1,\dots,F^k$ of $\Delta_{[2]^H}$ with $\Psi^{-1}(U_j)\cap(F^j)_i\ne\emptyset$ for all $j\in[k]$, we need show
\begin{equation}\label{e-4-lemma}
\cap_{j\in[k]}\Psi^{-1}(U_j)\cap(\overline{\cap_{j\in[k]}F^j})_i=\emptyset.
\end{equation}
If not, $\cap_{j\in[k]}U_j\cap\Psi((\overline{\cap_{j\in[k]}F^j})_i)\neq\emptyset$. We take $\mu\in \cap_{j\in[k]}U_j\cap\Psi((\overline{\cap_{j\in[k]}F^j})_i)$.

Since $\mu\in\cap_{j\in[k]}U_j$, note that $\sup_{B\in\beta}\diam(B,W_{G_n})<\varepsilon$ and $\Psi^{-1}(U_j)\cap(F^j)_i\ne\emptyset$ for all $j\in[k]$, by Claim 2 one has $$\varepsilon>W_{G_n}(\mu,\Psi((F^j)_i))\ge\delta\cdot\mu(S_{(\overline{F^j})_i}),\quad\forall j\in[k],$$ and it follows that $$W_{G_n}(\mu,\Psi((\cap_{j\in[k]}F^j)_i))\le\diam(X)\cdot\mu(S_{(\overline{\cap_{j\in[k]}F^j})_i})$$$$\le\diam(X)\cdot\sum_{F^j\;\text{distinct}}\mu(S_{(\overline{F^j})_i})\le\diam(X)\cdot2^H\cdot\frac\varepsilon\delta<\frac\delta2.$$

On the other hand, since $\mu\in\Psi((\overline{\cap_{j\in[k]}F^j})_i)$, for every $\theta\in \Psi((\cap_{j\in[k]}F^j)_i)$, by \eqref{WMn1} one has
$$W_{G_n}(\theta,\mu)\ge W(g_i\theta,g_i\mu)\ge\left\vert\int d(x,g_iS_{(\cap_{j\in[k]}F^j)_i})d(g_i\theta)-\int d(x,g_iS_{(\cap_{j\in[k]}F^j)_i})d(g_i\mu)\right\vert$$$$\ge d(g_iS_{(\cap_{j\in[k]}F^j)_i},g_iS_{(\overline{\cap_{j\in[k]}F^j})_i})\ge\delta.$$ 
And hence $$W_{G_n}(\mu,\Psi((\cap_{j\in[k]}F^j)_i))\ge\delta.$$ A contraicdtion. Thus \eqref{e-4-lemma} is proven.

Now by Lemma \ref{lem:lebesgue} $$\ord(\Psi^{-1}(\beta))\ge(2^H-1)\cdot m_n\ge(2^H-1)\cdot\frac{rT_n}{2}\ge(2^H-1)\cdot\frac{r}{2}\cdot(\frac{(1-\eta)\cdot\vert G_n\vert}{M^2}-1)$$$$>2^H\cdot\frac{r}{4}\cdot\frac{\vert G_n\vert}{4M^2}\ge2^H\cdot\frac{r}{4}\cdot\frac{\vert G_n\vert}{4(4H/r)^2}=2^H\cdot\frac{r^3}{4^4H^2}\cdot\vert G_n\vert.$$ So the claim is proven.
\end{proof}
Take a finite open cover $\alpha$ of $\cM(X)$ with $\sup_{A\in\alpha}\diam(A,W)<\varepsilon$. It follows from Claim 3 that for all sufficiently large $n\in\mathbb{N}$$$\mathcal{D}(\vee_{s\in G_n}s^{-1}\alpha\vert_{\widetilde{\pi}^{-1}(\delta_{y_{H,n}})})\ge\mathcal{D}(\vee_{s\in G_n}s^{-1}\alpha\vert_{L_n})\ge\frac{r^3}{4^4H^2}\cdot2^H\cdot\vert G_n\vert.$$ This ends the proof.


\begin{thebibliography}{9999999}

\bibitem[AKM65]{AKM65}
R. L. Adler, A. G. Konheim, M. H. McAndrew.
Topological entropy.
Trans. Amer. Math. Soc. 114(1965) 309--319.

\bibitem[BS75]{BS75}
W. Bauer, K. Sigmund.
Topological dynamics of transformations induced on the space of probability measures.
Monatsh. Math. 79(1975) 81--92.

\bibitem[Bow73]{Bow73}
R. Bowen.
Erratum to ``Entropy for group endomorphisms and homogeneous space''.
Trans. Amer. Math. Soc. 181(1973) 509--510.

\bibitem[BS25]{BS25}
D. Burguet, R. Shi.
Topological mean dimension of induced systems.
Trans. Amer. Math. Soc. 378(2025) 3085--3103.

\bibitem[DHP11]{DHP11}
D. Dou, W. Huang, K.K. Park.
Entropy dimension of topological dynamical systems.
Trans. Amer. Math. Soc. 363(2011) 659--680.

\bibitem[DHP19]{DHP19}
D. Dou, W. Huang, K.K. Park.
Entropy dimension of measure preserving systems.
Trans. Amer. Math. Soc. 371(2019) 7029--7065.

\bibitem[Dow11]{Dow11}
T. Downarowicz.
Entropy in dynamical systems.
Cambridge University Press (2011).

\bibitem[GTW00]{GTW00}
E. Glasner, J.P. Thouvenot, B. Weiss.
Entropy theory without a past.
Ergodic Theory Dynam. Systems 20(2000) 1355--1370.

\bibitem[GW95]{GW95}
E. Glasner, B. Weiss.
Quasi-factors of zero entropy systems.
J. Amer. Math. Soc. 8(1995) 665--686.

\bibitem[Gro99]{Gro99}
M. Gromov.
Topological invariants of dynamical systems and spaces of holomorphic maps: I.
Math. Phys. Anal. Geom. 2(1999) 323--415.

\bibitem[Gut15]{Gut15}
Y. Gutman.
Mean dimension and Jaworski-type theorems.
Proc. Lond. Math. Soc. 111(2015) 831--850.

\bibitem[GQT19]{GQT19}
Y. Gutman, Y. Qiao, M. Tsukamoto.
Application of signal analysis to the embedding problem of $\mathbb{Z}^k$-actions.
Geom. Funct. Anal. 29(2019) 1440--1502.

\bibitem[GT14]{GT14}
Y. Gutman, M. Tsukamoto.
Mean dimension and a sharp embedding theorem: extensions of aperiodic subshifts.
Ergodic Theory Dynam. Systems 34(2014) 1888--1896.

\bibitem[GT20]{GT20}
Y. Gutman, M. Tsukamoto.
Embedding minimal dynamical systems into Hilbert cubes.
Invent. math. 221(2020) 113--166.

\bibitem[HY06]{HY06}
W. Huang, X. Ye.
A local variational relation and applications.
Israel J. Math. 151(2006) 237--279.

\bibitem[HW41]{HW41}
W. Hurewicz, H. Wallman.
Dimension theory.
Princeton University Press (1941).

\bibitem[JQ24a]{JQ}
L. Jin, Y. Qiao.
Mean dimension of product spaces: a fundamental formula.
Math. Ann. 388(2024) 249--259.

\bibitem[JQ24b]{JQhilbert}
L. Jin, Y. Qiao.
The Hilbert cube contains a minimal subshift of full mean dimension.
Studia Math. 274(2024) 1--10.

\bibitem[KM78]{KM78}
M.G. Karpovsky, V.D. Milman. Coordinate density of sets of vectors. Discret. Math. 24(1978), 177--184.

\bibitem[KL05]{KL05}
D. Kerr, H. Li.
Dynamical entropy in Banach spaces.
Invent. math. 162(2005) 649--686.

\bibitem[KL07]{KL07}
D. Kerr, H. Li.
Independence in topological and C$^\ast$-dynamics.
Math. Ann. 338(2007) 869--926.

\bibitem[KL16]{KL16}
D. Kerr, H. Li.
Ergodic Theory. Independence and Dichotomies.
Springer Monographs in Mathematics. Springer, Cham (2016).
 
\bibitem[LL25]{LL25}
H. Li, K. Liu.
Local entropy theory, combinatorics, and local theory of Banach spaces.
arXiv:2507.03338 (2025).

\bibitem[Lia22]{Lia22}
B. Liang.
Conditional mean dimension.
Ergodic Theory Dynam. Systems 42(2022) 3152--3166.

\bibitem[Lin99]{Lin99}
E. Lindenstrauss.
Mean dimension, small entropy factors and an embedding theorem.
Inst. Hautes Études Sci. Publ. Math. 89(1999) 227--262.

\bibitem[LT14]{LT14}
E. Lindenstrauss, M. Tsukamoto.
Mean dimension and an embedding problem: an example.
Israel J. Math. 199(2014) 573--584.

\bibitem[LW00]{LW00}
E. Lindenstrauss, B. Weiss.
Mean topological dimension.
Israel J. Math. 115(2000) 1--24.

\bibitem[LW24]{LW24}
K. Liu, R. Wei.
Relative uniformly positive entropy of induced amenable group actions.
Ergodic Theory Dynam. Systems 44(2024) 569--593.

\bibitem[Par67]{Par67}
K. R. Parthasarathy.
Probability measures on metric spaces.
Academic Press, New York (1967).

\bibitem[QZ17]{QZ17}
Y. Qiao, X. Zhou.
Zero sequence entropy and entropy dimension.
Discrete Contin. Dyn. Syst. 37(2017) 435--448.

\bibitem[Sau72]{Sau72}
N. Sauer. On the density of families of sets. J. Comb. Theory Ser. A 13(1972), 145--147.

\bibitem[She72]{She72} 
S. Shelah. A combinatorial problem; stability and order for models and theories in infinitary languages. Pac. J. Math. 41(1972), 247--261.

\bibitem[SZ25]{SZ25}
R. Shi, G. Zhang.
Mean topological dimension of induced amenable group actions.
J. Differential Equations 427(2025) 827--842.

\bibitem[Tsu19]{Tsu19}
M. Tsukamoto.
Mean dimension of full shifts.
Israel J. Math. 230(2019) 183--193.

\bibitem[Tsu25]{Tsu25}
M. Tsukamoto.
On an analogue of the Hurewicz theorem for mean dimension.
Israel J. Math. (2025) https://doi.org/10.1007/s11856-025-2757-7

\end{thebibliography}
\end{document}